\documentclass[a4paper,12pt]{amsart}

\usepackage[utf8]{inputenc}
\usepackage{amsfonts,amssymb}
\usepackage{amsmath}
\usepackage{graphicx}

\usepackage{comment}
\usepackage[usenames,dvipsnames]{color}

\usepackage{a4wide}

\usepackage{empheq}
\numberwithin{equation}{section}

\usepackage{amsthm}
\theoremstyle{plain}
\newtheorem*{mainresult}{Theorem A}
\newtheorem{theorem}{Theorem}[section]
\newtheorem{lemma}[theorem]{Lemma}
\newtheorem{corollary}[theorem]{Corollary}

\newtheorem{proposition}[theorem]{Proposition}
\newtheorem{observation}[theorem]{Observation}

\newtheorem*{lemma*}{Lemma}

\theoremstyle{definition}
\newtheorem{definition}[theorem]{Definition}
\newtheorem{remark}[theorem]{Remark}

\newcommand{\R}{\mathbb{R}}
\newcommand{\N}{\mathbb{N}}
\newcommand{\Z}{\mathbb{Z}}

\newcommand{\calP}{\mathcal{P}}

\newcommand{\defeq}{\mathrel{\mathop{:}}=}

\newcommand{\abs}[1]{\lvert #1 \rvert}
\newcommand{\gen}[1]{\langle #1 \rangle}
\newcommand{\floor}[1]{\left\lfloor #1 \right\rfloor}
\newcommand{\ceil}[1]{\left\lceil #1 \right\rceil}
\newcommand{\lk}{\operatorname{lk}}
\newcommand{\st}{\operatorname{st}}
\newcommand{\Hom}{\operatorname{Hom}}
\newcommand{\CAT}{\operatorname{CAT}}
\newcommand{\F}{\operatorname{F}}

\newcommand{\id}{\operatorname{id}}

\newcommand{\match}{\mathcal{M}}
\newcommand{\gmatch}{\mathcal{GM}}
\newcommand{\nerve}{\mathcal{N}}

\newcommand{\elsplit}{\Lambda}
\newcommand{\elmerge}{\mathrm{V}}
\newcommand{\elnothing}{\mathrm{I}}

\newcommand{\into}{\hookrightarrow}

\numberwithin{equation}{section}

\begin{document}

\title[$\Sigma$-invariants of $F$]{The $\Sigma$-invariants of Thompson's group $F$\\via Morse Theory}
\dedicatory{Dedicated to Ross Geoghegan on the occasion of his 70th birthday}
\date{\today}


\author[S.~Witzel]{Stefan Witzel}
\address{Department of Mathematics, Bielefeld University, PO Box 100131, 33501 Bielefeld, Germany}
\email{switzel@math.uni-bielefeld.de}

\author[M.~C.~B.~Zaremsky]{Matthew C.~B.~Zaremsky}
\address{Department of Mathematical Sciences, Binghamton University, 
Binghamton, NY 13902}
\email{zaremsky@math.binghamton.edu}

\begin{abstract}
 Bieri, Geoghegan and Kochloukova computed the BNSR-invariants $\Sigma^m(F)$ of Thompson's group $F$ for all $m$. We recompute these using entirely geometric techniques, making use of the Stein--Farley $\CAT(0)$ cube complex $X$ on which $F$ acts.
\end{abstract}

\maketitle
\thispagestyle{empty}


\section*{Introduction}

In \cite{bieri10}, Bieri, Geoghegan and Kochloukova computed $\Sigma^m(F)$ for all $m$. Here $F$ is Thompson's group, and $\Sigma^m$ is the $m$th \emph{Bieri--Neumann--Strebel--Renz (BNSR) invariant}. This is a topological invariant of a group of type $\F_m$ \cite{bieri87,bieri88}. The proof in \cite{bieri10} makes use of various algebraic facts about $F$, for instance that it contains no non-abelian free subgroups, and that it is isomorphic to an ascending HNN-extension of an isomorphic copy of itself, and applies tools specific to such groups to compute all the $\Sigma^m(F)$.

In this note we consider the free action of $F$ by isometries on a proper $\CAT(0)$ cube complex $X$, the \emph{Stein--Farley complex}. We apply a version of Bestvina--Brady Morse theory to this space, and recompute the invariants $\Sigma^m(F)$. The crucial work to do, using this approach, is to analyze the homotopy type of \emph{ascending links} of vertices in $X$ and \emph{superlevel sets} in $X$ with respect to various \emph{character height functions}.

The abelianization of Thompson's group $F$ is free abelian of rank $2$. A basis of $\Hom(F,\R) \cong \R^2$ is given by two homomorphisms usually denoted $\chi_0$ and $\chi_1$. The main result of \cite{bieri10} is:

\begin{mainresult}\label{thrm:main}
 Write a non-trivial character $\chi \colon F \to \R$ as $\chi = a \chi_0 + b\chi_1 $ (with $(a,b) \ne (0,0)$). Then $[\chi]$ is in $\Sigma^1(F)$ unless $a > 0$ and $b=0$, or $b>0$ and $a=0$. Moreover, $[\chi]$ is in $\Sigma^2(F) = \Sigma^\infty(F)$ unless $a \ge 0$ and $b \ge 0$.
\end{mainresult}

Here $[\chi]$ denotes the equivalence class of $\chi$ up to positive scaling.

The proof in \cite{bieri10} is partly algebraic, using the fact that $F$ is an ascending HNN-extension of itself, and that it contains no non-abelian free subgroups. Here we give a completely geometric, self-contained proof of Theorem~A. It is possible is that this approach could be useful in determining the BNSR-invariants of other interesting groups.

In Section~\ref{sec:invariants} we recall the invariants $\Sigma^m(G)$, and set up the Morse theory tools we will use. In Section~\ref{sec:group_and_space} we recall Thompson's group $F$ and the Stein--Farley complex $X$, and discuss characters of $F$ and height functions on $X$. Section~\ref{sec:links_subcomplexes} is devoted to combinatorially modeling vertex links in $X$, which is a helpful tool in the computations of the $\Sigma^m(F)$. Theorem~A is proven in stages in Sections~\ref{sec:long_interval} through~\ref{sec:short_interval}.

\subsection*{Acknowledgments} Thanks are due to Robert Bieri for suggesting this problem to us, and to Matt Brin for helpful discussions.

The first author gratefully acknowledges support by the DFG through the project WI 4079/2 and through the SFB 701.


\section{The invariants}\label{sec:invariants}

A \emph{character} of a group $G$ is a homomorphism $\chi\colon G\to \R$. If the image is infinite cyclic, the character is \emph{discrete}. If $G$ is finitely generated, we can take $\Hom(G,\R) \cong \R^d$ and mod out scaling by positive real numbers to get $S(G)\defeq S^{d-1}$, called the \emph{character sphere}. Here $d$ is the rank of the abelianization of $G$. Now, the definition of the \emph{Bieri--Neumann--Strebel (BNS) invariant} $\Sigma^1(G)$ for $G$ finitely generated (first introduced in \cite{bieri87}) is the subset of $S(G)$ defined by:
\[
\Sigma^1(G) \defeq \{[\chi]\in S(G)\mid \Gamma_{0\le\chi} \text{ is connected}\} \text{.}
\]
Here $\Gamma$ is the Cayley graph of $G$ with respect to some finite generating set, and $\Gamma_{0\le\chi}$ is the full subgraph spanned by those vertices $g$ with $0\le\chi(g)$.

The higher \emph{Bieri--Neumann--Strebel--Renz (BNSR) invariants} $\Sigma^m(G)$ ($m\in\N\cup\{\infty\}$), introduced in \cite{bieri88}, are defined somewhat analogously. Let $G$ be of type $\F_m$ and let $\chi$ be a character of $G$. Let $\Gamma^m$ be the result of equivariantly gluing cells, up to dimension $m$, to the Cayley graph $\Gamma$, in a $G$-cocompact way to produce an $(m-1)$-connected space (that this is possible is more or less the definition of being of type $\F_m$). For $t\in\R$ let $\Gamma^m_{t\le \chi}$ be the full subcomplex spanned by those $g$ with $t\le \chi(g)$. Then $[\chi]\in\Sigma^m(G)$ by definition if the filtration $(\Gamma^m_{t\le\chi})_{t\in\R}$ is essentially $(m-1)$-connected.

Recall that a space $Y$ is $(m-1)$-connected if every continuous map $S^{k-1} \to Y$ is homotopic to a constant map for $k \le m$. A filtration $(Y_t)_{t \in \R}$ (with $Y_t \supseteq Y_{t+1}$) is essentially $(m-1)$-connected if for every $t$ there is a $t' \le t$ such that for all $k\le m$, every continuous map $S^{k-1} \to Y_t$ is homotopic in $Y_{t'}$ to a constant map. Recall also that a pair $(Y,Y_0)$ is $(m-1)$-connected if for all $k\le m$, every map $(D^{k-1},S^{k-2}) \to (Y,Y_0)$ is homotopic relative $Y_0$ to a map with image in $Y_0$.

In the realm of finiteness properties, the main application of the BNSR-invariants is the following, see \cite[Theorem~5.1]{bieri88} and \cite[Theorem~1.1]{bieri10}:

\begin{theorem}\cite[Theorem~1.1]{bieri10}
 Let $G$ be a group of type $\F_m$ and $N$ a normal subgroup of $G$ with $G/N$ abelian. Then $N$ is of type $\F_m$ if and only if for every $\chi\in\Hom(G,\R)$ with $\chi(N)=0$ we have $[\chi]\in\Sigma^m(G)$.
\end{theorem}
As an example, if $\chi$ is discrete then $\ker(\chi)$ is of type $\F_m$ if and only if $[\pm\chi]\in\Sigma^m(G)$.

It turns out that to compute $\Sigma^m(G)$, one can use spaces and filtrations other than just $\Gamma^m$ and $(\Gamma^m_{t\le\chi})_{t\in\R}$, namely:

\begin{proposition}\cite[Definition~8.1]{bux04}\label{prop:bnsr_def}
 Let $G$ be of type $\F_m$, acting cellularly on an $(m-1)$-connected CW complex $Y$. Suppose the stabilizer of any $k$-cell is of type $\F_{m-k}$, and that the action of $G$ on $Y^{(m)}$ is cocompact. For any non-trivial character $\chi\in\Hom(G,\R)$, there is a \emph{character height function}, denoted $h_\chi$, i.e., a continuous map $h_\chi\colon Y\to\R$, such that $h_\chi(gy)=\chi(g)+h_\chi(y)$ for all $y\in Y$ and $g\in G$. Then $[\chi]\in\Sigma^m(G)$ if and only if the filtration $(h_\chi^{-1}([t,\infty)))_{t\in\R}$ is essentially $(m-1)$-connected.
\end{proposition}

It is a fact that this is independent of $Y$ and $h_\chi$. In our case $G$ can be naturally regarded as a subset of $Y$ and we will therefore write $\chi$ to denote both the character and the character height function.

The filtration of $Y$ that we will actually use in practice is $(Y_{t\le \chi})_{t\in\R}$, where $Y_{t\le \chi}$ is defined to be the full subcomplex of $Y$ supported on those vertices $y$ with $t\le \chi(y)$. Our working definition of $\Sigma^m(G)$ will be the following:

\begin{corollary}[Working definition]\label{cor:changing_filtrations}
 Let $G$, $Y$ and $\chi$ be as in Proposition~\ref{prop:bnsr_def}. Then $[\chi]\in\Sigma^m(G)$ if and only if the filtration $(Y_{t\le \chi})_{t\in\R}$ is essentially $(m-1)$-connected.
\end{corollary}

\begin{proof}
 We need to show that $(Y_{t\le \chi})_{t\in\R}$ is essentially $(m-1)$-connected if and only if $(h_\chi^{-1}([t,\infty)))_{t\in\R}$ is. First note that this is really a statement about $m$-skeleta. Then since $m$-skeleta are $G$-cocompact, there exists a uniform bound $D$ such that for any two points $y,y'$ of $Y$ sharing a $k$-cell ($k\le m$), we have $\abs{\chi(y)-\chi(y')}\le D$. Now the result is immediate, since $Y_{t\le \chi}^{(m)} \subseteq \chi^{-1}([t-D,\infty))^{(m)}$ and $\chi^{-1}([t,\infty))^{(m)} \subseteq Y_{t-D\le \chi}^{(m)}$.
\end{proof}

\begin{remark}
 Bieri and Geoghegan generalized the BNSR-invariants in \cite{bieri03}, to a family of invariants $\Sigma^m(\rho)$ defined for any (sufficiently nice) action $\rho$ of a group on a $\CAT(0)$ space. The classical BNSR-invariants agree with the case where $\rho$ is the action of $G$ on $\Hom(G,\R)$. See also \cite[Section~18.3]{geoghegan08}.
\end{remark}

\subsection{Morse theory}\label{sec:morse}

The criterion Corollary~\ref{cor:changing_filtrations} is particularly useful in the situation where $Y$ is an affine cell complex and $\chi$ is affine on cells. One can then make use of a version of Bestvina--Brady Morse theory and study relative connectivity locally in terms of ascending/descending links.

By an \emph{affine cell complex} we mean a complex that is obtained by gluing together euclidean polytopes. More precisely, $Y$ is an affine cell complex if it is the quotient $\hat{Y}/\mathop{\sim}$ of a disjoint union $\hat{Y} = \bigcup_\lambda C_{\lambda}$ of euclidean polytopes modulo an equivalence relation $\sim$ such that every polytope is mapped injectively to $Y$, and such that if two faces of polytopes have a (relative) interior point identified then their entire (relative) interior is isometrically identified (see \cite[Definition~I.7.37]{bridson99} for a more general definition). In particular, every cell (meaning every image of some polytope in $\hat{Y}$) carries an affine structure. The link $\lk_Y v$ of a vertex $v$ of $Y$ consists of directions issuing at the vertex. It naturally carries the structure of a spherical cell complex, whose closed cells consist of directions that point into closed cells of $Y$.

\begin{definition}[Morse function]\label{def:morse}
 The most general kind of \emph{Morse function} on $Y$ that we will be using is a map $(h,s) \colon Y \to \R \times \R$ such that both $h$ and $s$ are affine on cells. The codomain is ordered lexicographically, and the conditions for $(h,s)$ to be a Morse function are the following: the function $s$ takes only finitely many values on vertices of $Y$, and there is an $\varepsilon > 0$ such that every pair of adjacent vertices $v$ and $w$ either satisfy $\abs{h(v) - h(w)} \ge \varepsilon$, or else $h(v) = h(w)$ and $s(v)\ne s(w)$.
\end{definition}

As an example, if $h$ is discrete and $s$ is constant, we recover the notion of ``Morse function'' from \cite{bestvina97}. We think of a Morse function as assigning a \emph{height} $h(v)$ to each vertex $v$. The secondary height function $s$ ``breaks ties'' for adjacent vertices of the same height. We will speak of $(h,s)(v)$ as the \emph{refined height} of $v$. Note that since $h$ is affine, the set of points of a cell $\sigma$ on which $h$ attains its maximum is a face $\bar{\sigma}$. Since $s$ is affine as well, the set of points of $\bar{\sigma}$ on which $s$ attains its maximum is a face $\hat{\sigma}$. If $\hat{\sigma}$ were to have two adjacent vertices, these vertices would have the same refined height, and $(h,s)$ would not be a Morse function. This shows that every cell has a unique vertex of maximal refined height and (by symmetry) a unique vertex of minimal refined height.

The \emph{ascending star} $\st^{(h,s)\uparrow} v$ of a vertex $v$ (with respect to $(h,s)$) is the subcomplex of $\st v$ consisting of cells $\sigma$ such that $v$ is the vertex of minimal refined height in $\sigma$. The \emph{ascending link} $\lk^{(h,s)\uparrow} v$ of $v$ is the link of $v$ in $\st^{(h,s)\uparrow} v$. The \emph{descending star} and the \emph{descending link} are defined analogously. A consequence of $h$ and $s$ being affine is the following.

\begin{observation}\label{obs:ascending_links_full}
 Ascending and descending links are full subcomplexes. \qed
\end{observation}

We use notation like $Y_{p \le h \le q}$ to denote the full subcomplex of $Y$ supported on the vertices $v$ with $p \le h(v) \le q$ (this is the union of the closed cells all of whose vertices lie within the bounds). An important tool we will use is:

\begin{lemma}[Morse Lemma]
 Let $p,q,r \in \R \cup \{\pm \infty\}$ be such that $p \le q \le r$. If for every vertex $v \in Y_{q < h \le r}$ the descending link $\lk^{(h,s)\downarrow}_{Y_{p \le h}} v$ is $(k-1)$-connected then the pair $(Y_{p \le h \le r},Y_{p \le h \le q})$ is $k$-connected. If for every vertex $v \in Y_{p \le h < q}$ the ascending link $\lk^{(h,s)\uparrow}_{Y_{h \le r}} v$ is $(k-1)$-connected then the pair $(Y_{p \le h \le r},Y_{q \le h \le r})$ is $k$-connected.
\end{lemma}

\begin{proof}
 The second statement is like the first with $(h,s)$ replaced by $-(h,s)$, so we only prove the first. Using induction (and compactness of spheres in case $r = \infty$) we may assume that $r-q \le \varepsilon$ (where $\varepsilon > 0$ is as in Definition~\ref{def:morse}). By compactness of spheres, it suffices to show that there exists a well order $\preceq$ on the vertices of $Y_{q < h \le r}$ such that the pair
 \[
 (S_{\preceq v},S_{\prec v}) \defeq \left(Y_{p \le h \le q} \cup \bigcup_{w \preceq v} \st^{(h,s)\downarrow}_{Y_{p \le h}} w \text{, } Y_{p \le h \le q} \cup \bigcup_{w \prec v} \st^{(h,s)\downarrow}_{Y_{p \le h}} w\right)
 \]
 is $k$-connected for every vertex $v\in Y_{q < h \le r}$. To this end, let $\preceq$ be any well order satisfying $v\prec v'$ whenever $s(v)<s(v')$ (this exists since $s$ takes finitely many values on vertices). Note that $S_{\preceq v}$ is obtained from $S_{\prec v}$ by coning off $S_{\prec v} \cap \partial \st v$. We claim that this intersection is precisely the boundary of $\st^{(h,s)\downarrow}_{Y_{p \le h}} v$ in $Y_{p\le h}^{(h,s)\le (h,s)(v)}$, which we will denote $B$, and which is homeomorphic to $\lk^{(h,s)\downarrow}_{Y_{p \le h}} v$ and hence $(k-1)$-connected by assumption. The inclusion $S_{\prec v} \cap \partial \st v \subseteq B$ is clear. Since $S_{\prec v} \cap \partial \st v$ is a full subcomplex of $\partial \st v$, it suffices for the converse to verify that any vertex $w$ adjacent to $v$ with $(h,s)(w) < (h,s)(v)$ lies in $S_{\prec v}$. If $h(w) < h(v)$ then $h(w) \le h(v) - \varepsilon \le r - \varepsilon \le q$, so $w \in Y_{p \le h \le q}$. Otherwise $s(w) < s(v)$ and thus $w \prec v$.
\end{proof}

\subsection{Negative properties}\label{sec:negative_properties}

When trying to disprove finiteness properties using Morse theory, one often faces the following problem: suppose the Morse function is on a contractible space and the ascending links are always $(m-1)$-connected but infinitely often not $m$-connected. One would like to say that every ascending link that is not $m$-connected cones off at least one previously non-trivial $m$-sphere, and thus there are $m$-spheres that only get coned off arbitrarily late (and hence the filtration is not essentially $m$-connected). But this argument does not work in general because it is possible that the $m$-sphere one is coning off was actually already homotopically trivial in the superlevel set, and then one is actually \emph{producing} an $(m+1)$-sphere. This second option can be excluded if one can make sure that no $(m+1)$-spheres are ever coned off (for example, if the whole contractible space is $(m+1)$-dimensional) and then the argument works. In general though, the difference between killing $m$-spheres and producing $(m+1)$-spheres is not visible locally and one has to take a more global view. The following will be useful in doing so.

\begin{observation}\label{obs:weak_bottleneck_trick}
 Let an $(m-1)$-connected affine cell complex $X$ be equipped with a Morse function $(h,s)\colon X \to \R\times\R$ and assume that all ascending links are $(m-2)$-connected. Then the filtration $(X_{t \le h})_{t \in \R}$ is essentially $(m-1)$-connected, if and only if $X_{p \le h}$ is $(m-1)$-connected for some $p$, if and only if all $X_{p' \le h}$ are $(m-1)$-connected for all $p'\le p$ (for some $p$).
\end{observation}

\begin{proof}
 Since ascending links are $(m-2)$-connected, the Morse Lemma implies that for any $p<q$ the pairs $(X_{p\le h},X_{q \le h})$ are $(m-1)$-connected, so in particular the map $\pi_k(X_{q \le h} \into X_{p \le h})$ is an isomorphism for $k<m-1$ and surjective for $k=m-1$. Now for these $p$ and $q$, we see that this map induces the trivial map in homotopy up to dimension $m-1$ if and only if the homotopy groups vanish on $X_{p \le h}$. We conclude that the filtration is essentially $(m-1)$-connected if and only if $X_{p \le h}$ is $(m-1)$-connected for some $p \in \Z$, or equivalently all $p' \le p$.
\end{proof}


\section{Thompson's group and the Stein--Farley complex}\label{sec:group_and_space}

Thompson's group $F$ has appeared over the past decades in a variety of situations, and has proved to have many strange and interesting properties. It was the first example of a torsion-free group of infinite cohomological dimension that is of type $\F_\infty$ \cite{brown84}. Since it is of type $\F_\infty$, one can ask what its BNSR-invariants $\Sigma^m(F)$ are, for arbitrary $m$. The $m=1$ case was answered by Bieri, Neumann and Strebel in \cite{bieri87}, and the $m\ge 2$ case by Bieri, Geoghegan and Kochloukova in \cite{bieri10}. The main result of the present work is a recomputation of the $\Sigma^m(F)$, making use of a $\CAT(0)$ cube complex on which $F$ acts freely, called the \emph{Stein--Farley complex} $X$.

\subsection{The group}\label{sec:group}

The fastest definition of $F$ is via its standard infinite presentation,
\[
F=\gen{x_i, i\in\N \mid x_j x_i = x_i x_{j+1}, i<j}\text{.}
\]
One can also realize it as the group of orientation-preserving piecewise linear homeomorphisms of the interval $[0,1]$ with dyadic slopes and breakpoints. For our purposes, the most useful definition is in terms of \emph{split-merge tree diagrams}.

A \emph{split-merge tree diagram} $(T_-/T_+)$ consists of a binary tree $T_-$ of ``splits'' and a binary tree $T_+$ of ``merges'', such that $T_-$ and $T_+$ have the same number of leaves. The leaves of $T_-$ and of $T_+$ are naturally ordered left to right and we identify them. Two split-merge tree diagrams are \emph{equivalent} if they can be transformed into each other via a sequence of reductions or expansions. A \emph{reduction} is possible if $T_-$ and $T_+$ contain terminal carets whose leaves coincide. The reduction consists of deleting both of these carets. An expansion is the inverse of a reduction. We denote the equivalence class of $(T_-/T_+)$ by $[T_-/T_+]$.

These $[T_-/T_+]$ are the elements of $F$. The multiplication, say of elements $[T_-/T_+]$ and $[U_-/U_+]$, written $[T_-/T_+] \cdot [U_-/U_+]$, is defined as follows. First note that $T_+$ and $U_-$ admit a binary tree $S$ that contains them both, so using expansions we have $[T_-/T_+] = [\hat{T}_-/S]$ and $[U_-/U_+] = [S/\hat{U}_+]$ for some $\hat{T}_-$ and $\hat{U}_+$. Then we can define
\[
[T_-/T_+] \cdot [U_-/U_+] \defeq [\hat{T}_-/S] \cdot [S/\hat{U}_+] = [\hat{T}_-/\hat{U}_+] \text{.}
\]
This multiplication is well defined, and it turns out that the resulting structure is a group, namely $F$. More information on the background of $F$ can be found in \cite{cannon96}.

\subsection{The Stein--Farley complex}\label{sec:stein_farley}

We now recall the Stein--Farley cube complex $X$ on which $F$ acts. This was first constructed by Stein in \cite{stein92}, and shown to be $\CAT(0)$ by Farley \cite{farley03}. We begin by generalizing split-merge tree diagrams to allow for \emph{forests}: A \emph{split-merge diagram} $(E_-/E_+)$ consists of a binary forest $E_-$ of ``splits'' and a binary forest $E_+$ of ``merges'' such that $E_-$ and $E_+$ have the same number of leaves. By a \emph{binary forest} we mean a finite sequence of rooted binary trees. Thus the leaves of $E_-$ and of $E_+$ are naturally ordered left to right and we identify them. As in Figure~\ref{fig:split-merge_diagram} we will usually draw $E_+$ upside down. We call the roots of $E_-$ \emph{heads} and the roots of $E_+$ \emph{feet} of the diagram.

\begin{figure}
\centering
\includegraphics{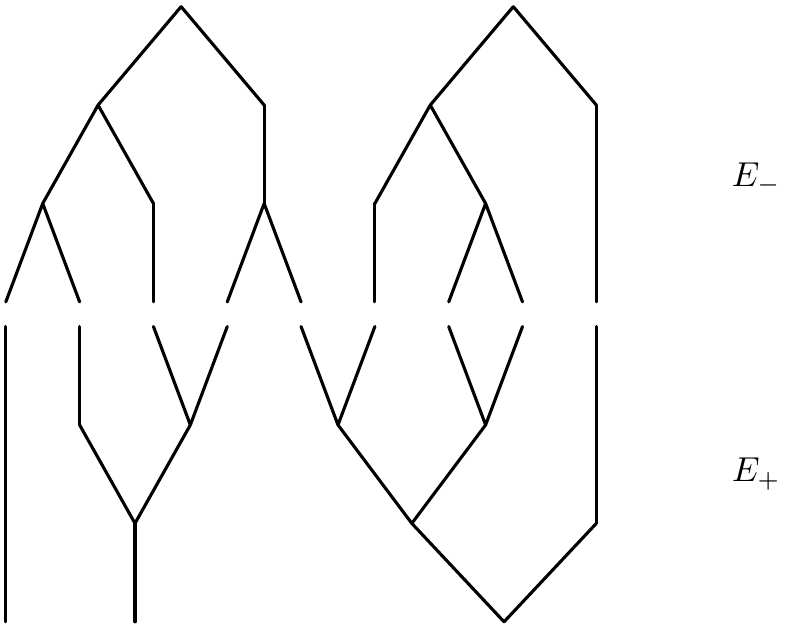}
\caption{A split-merge diagram with $2$ heads and $3$ feet. The diagram is not reduced because the $7$th and $8$th leaf of $E_-$ and $E_+$ both lie in terminal carets.}
\label{fig:split-merge_diagram}
\end{figure}

Just like the tree case, we have a notion of equivalence using reduction and expansion, defined the same way. It is easy to see that every equivalence class contains a unique reduced split-merge diagram: just note that if a diagram $(E_-/E_+)$ has two possible reductions that lead to diagrams $(D^1_-/D^1_+)$ and $(D^2_-/D^2_+)$, both reductions can be performed at once to give a diagram $(C_-/C_+)$ that is a common reduction of $(D^1_-/D^1_+)$ and $(D^2_-/D^2_+)$ (i.e., reduction is confluent in a very strong way). We will sometimes abuse language and speak of a split-merge diagram, when really we are talking about equivalence classes of split-merge diagrams.

Let $\calP$ be the set of equivalence classes of split-merge diagrams. This set has two important pieces of structure.

\textbf{Groupoid.} The first is a groupoid structure. If $[E_-/E_+]$ has $k$ heads and $\ell$ feet, and $[D_-/D_+]$ has $\ell$ heads and $m$ feet, then we will define their product, written $[E_-/E_+] \cdot [D_-/D_+]$, which is a split-merge diagram with $k$ heads and $m$ feet. Like in $F$, with split-merge tree diagrams, one way to define the product is to first note that $E_+$ and $D_-$ admit a binary forest $C$ that contains them both. Using expansions one can thus write $[E_-/E_+] = [\hat{E}_-/C]$ and $[D_-/D_+] = [C/\hat{D}_+]$ for some $\hat{E}_-$ and $\hat{D}_+$. Now we define
\[
[E_-/E_+] \cdot [D_-/D_+] \defeq [\hat{E}_-/C] \cdot [C/\hat{D}_+] = [\hat{E}_-/\hat{D}_+] \text{.}
\]
There is also a more visual description. The product can be obtained by stacking the diagram $(E_-/E_+)$ on top of $(D_-/D_+)$ and then applying a sequence of operations that are dual to expansion/reduction, namely if a merge is immediately followed by a split, both can be removed. Successively applying this operation eventually leads to a split-merge diagram that represents $[E_-/E_+]\cdot [D_-/D_+]$.

For this multiplication operation to give $\calP$ a groupoid structure, we need identities and inverses. A forest in which all trees are trivial is called a \emph{trivial forest}. The trivial forest with $n$ trees is denoted $\id_n$. We consider binary forests as split-merge diagrams via the embedding $E \mapsto [E/\id_m]$ where $m$ is the number of leaves of $E$. In particular for every $m$ we have the element $[\id_m/\id_m]$, and this is clearly a multiplicative identity for split-merge diagrams against which it can be multiplied. Inverses are straightforward: the (left and right) inverse of $[E_-/E_+]$ is $[E_+/E_-]$.

\begin{observation}
 $\calP$ is a groupoid with the above multiplication.
\end{observation}

Since $F$ lives in $\calP$ as the set of elements with one head and one foot, we have an action of $F$, by multiplication, on the subset $\calP_1$ of elements with one head.

The second piece of structure on $\calP$ is an order relation.

\textbf{Poset.} The order is defined by: $[E_-/E_+] \le [D_-/D_+]$ whenever there is a binary forest $C$ such that $[E_-/E_+] \cdot C = [D_-/D_+]$ (recall that we identified the binary forest $C$ with a split-merge diagram $[C/\id_m]$). In words, $[D_-/D_+]$ is greater than $[E_-/E_+]$ if it can be obtained from it by splitting feet. It is straightforward to check that $\le$ is a partial order. The subset $\calP_1$ of elements with one head is a subposet.

The topological realization of the poset $(\calP_1,\le)$ is a simplicial complex on which $F$ acts, and the \emph{Stein--Farley complex} $X$ is a certain invariant subcomplex with a natural cubical structure. Given split-merge diagrams $[E_-/E_+] \le [E_-/E_+] \cdot E$, we write $[E_-/E_+] \preceq [E_-/E_+] \cdot E$ if $E$ is an \emph{elementary forest}. This means that each of its trees is either trivial, or a single caret. Now $X$ is defined to be the subcomplex of $|\calP_1|$ consisting of those chains $x_0<\cdots<x_k$ with $x_i \preceq x_j$ for all $i\le j$. The cubical structure is given by intervals: given $x\preceq y$, the interval $[x,y]\defeq\{z\mid x\le z\le y\}$ is a Boolean lattice of dimension $n$, and so the simplices in $[x,y]$ glom together into an $n$-cube. Here $n$ is the number of carets in $E$, with $y=x\cdot E$.

\begin{theorem}\cite{farley03}
 $X$ is a $\CAT(0)$ cube complex.
\end{theorem}

Note that the action of $F$ on $X$ is free. It is free on vertices since the action is just by multiplication in a groupoid. Also, it is free on cubes since if a group element stabilizes $[x,y]$ it must fix $x$ and $y$.

Every cube $\sigma$ has a unique vertex $x$ with fewest feet and a unique vertex $y$ with most feet. There is a unique elementary forest $E$ with $y=x\cdot E$, and the other vertices of $\sigma$ are obtained by multiplying $x$ by subforests of $E$. We introduce some notation for this: suppose $x$ has $k$ feet and $E=(A_1,\dots,A_k)$, where each $A_i$ is either $\elnothing$ or $\elsplit$; here $\elnothing$ is the trivial tree and $\elsplit$ is the tree with one caret. Let $\Phi$ be the set of subforests of $E$, written $\Phi \defeq \gen{A_1,\ldots,A_k}$. Then the vertex set of $\sigma$ is precisely $x\Phi$.

If we take a different vertex $z$ of $\sigma$ as ``basepoint'', then we also have to allow merges. Say $z$ has $m > k$ feet. Then we can write $\sigma = z\Psi$ where $\Psi$ is now of the form $\gen{A_1,\ldots,A_m}$ where each $A_i$ is either $\elnothing$, $\elsplit$ or $\elmerge$. Here $\elmerge$ is the inverse of the tree with one caret (so an upside-down caret). The tuple $(A_1,\ldots,A_m)$ is now to be thought of as a split-merge diagram, and the set $\Psi$ as the set of all split-merge diagrams that can be obtained by removing some of the carets. As before, the vertex set of $\sigma$ is $z\Psi$.

\subsection{Characters and affine character height functions}\label{sec:characters}

It is well known that $\Hom(F,\R)\cong \R^2$. A standard choice of basis is $\{\chi_0,\chi_1\}$, where $\chi_0$ and $\chi_1$ are most easily described when viewing elements of $F$ as piecewise linear homeomorphisms of $[0,1]$ with dyadic slopes and breakpoints. Then $\chi_0(f)\defeq \log_2(f'(0))$ and $\chi_1(f)\defeq \log_2(f'(1))$. Here the derivatives are taken on the right for $\chi_0$ and the left for $\chi_1$. So any character of $F$ is of the form $\chi=a\chi_0+b\chi_1$ for $a,b\in\R$.

For a tree $T$ we define $L(T)$ to be the number of carets above the leftmost leaf of $T$, and $R(T)$ to be the number of carets above the rightmost leaf. Then the characters $\chi_i\colon F\to\Z$ can be expressed in terms of split-merge tree diagrams as
\begin{equation}
\label{eq:chi_i}
\chi_0([T_-/T_+]) \defeq L(T_+) - L(T_-) \text{ and } \chi_1([T_-/T_+]) \defeq R(T_+) - R(T_-)\text{.}
\end{equation}
It is readily checked that~\eqref{eq:chi_i} is invariant under the equivalence relation on split-merge tree diagrams, and thus gives well defined maps. Replacing binary trees by binary forests, the above definition generalizes verbatim to arbitrary split-merge diagrams. In particular, the $\chi_i$ can now be evaluated on vertices of $X$. Moreover, any character $\chi$ on $F$ can be written as a linear combination
\begin{equation}
\label{eq:character_linear_combination}
\chi = a \chi_0 + b\chi_1
\end{equation}
and thus extends to arbitrary split-merge diagrams by interpreting \eqref{eq:character_linear_combination} as a linear combination of the extended characters.

Since $\chi$ will be our height function, we need the following.

\begin{lemma}\label{lem:affine_extend}
 Any character $\chi$ extends to an affine map $\chi \colon X \to \R$.
\end{lemma}

\begin{proof}
 It suffices to show that $\chi_0$, $\chi_1$ can be affinely extended. By symmetry it suffices to treat $\chi_0$. Let $\square_2=v \Phi$ be a square, say $\Phi=\gen{A_1,\dots,A_k}$, with exactly two $A_i$ being $\elsplit$ and all others being $\elnothing$. Say $A_i=A_j=\elsplit$ for $i<j$. Now either $i > 1$ and $\chi_0$ is constant on $\square_2$, or $i = 1$ and $\chi_0$ is affine-times-constant on $\square_2$. We conclude using Lemma~\ref{lem:affine_on_cubes} below.
\end{proof}

\begin{lemma}\label{lem:affine_on_cubes}
 A map $\varphi \colon \{0,1\}^n \to \R$ can be affinely extended to the cube $[0,1]^n$ if it can be affinely extended to its $2$-faces.
\end{lemma}

\begin{proof}
 The values of $\varphi$ on the zero vector and on the standard basis vectors define a unique affine map $\tilde{\varphi}$. The goal is therefore to show that $\tilde{\varphi}$ coincides with $\varphi$ on all the other vertices of $[0,1]^n$. This is proved for $v \in \{0,1\}^n$ by induction on the number of entries in $v$ equal to $1$. Let $v = (v_i)_{1 \le i \le n}$ with $v_i = v_j = 1$ for some $i \ne j$. We know by induction that $\varphi(w) = \tilde{\varphi}(w)$ for the three vertices $w$ obtained from $v$ by setting to $0$ the entries with index $i$ or $j$ or both. But these three vertices together with $v$ span a $2$-face, and so $\tilde{\varphi}(v)$ is the value of the (unique) affine extension of $\varphi$ to that $2$-face. Thus $\varphi(v) = \tilde{\varphi}(v)$.
\end{proof}

These extended characters $\chi$ will be our height functions. Our secondary height will be given by the number of feet function or its negative.

\begin{observation}\label{obs:feet_affine}
 There is a map $f \colon X \to \R$ that is affine on cubes and assigns to any vertex its number of feet.\qed
\end{observation}

Since our definition of Morse function required the secondary height function to take only finitely many values on vertices, for the next proposition we must restrict to subcomplexes of the form $X_{p\le f\le q}$. This is the full subcomplex supported on those vertices $v$ with $p\le f(v)\le q$.

\begin{proposition}\label{prop:char_morse}
 Let $\chi$ be a character. The pair $(\chi,f)$, as well as the pair $(\chi,-f)$, is a Morse function on $X_{p\le f\le q}$, for any $p\le q<\infty$.
\end{proposition}

\begin{proof}
 We have already seen in Lemma~\ref{lem:affine_extend} and Observation~\ref{obs:feet_affine} that $\chi$ and $f$ are affine. Also, $f$ takes finitely many values on vertices in $X_{p\le f\le q}$. It remains to see that there is an $\varepsilon > 0$ such that any two adjacent vertices $x$ and $x'$ either have $\abs{\chi(x) - \chi(x')} \ge \varepsilon$, or else $\chi(x)=\chi(x')$ and $f(x) \ne f(x')$. Let $\varepsilon=\min\{\abs{c}\mid c\in\{a,b\}\setminus\{0\}\}$, where $\chi=a\chi_0+b\chi_1$. We obtain $x'$ from $x$ by adding one split or one merge to the feet of $x'$. If it does not involve the first or last foot, then $\chi(x')=\chi(x)$. Otherwise $\chi(x')=\chi(x)\pm c$ for $c\in\{a,b\}$, and so $\abs{\chi(x) - \chi(x')} = \abs{c}$, which is either $0$ or at least $\varepsilon$.
\end{proof}


\section{Links and subcomplexes}\label{sec:links_subcomplexes}

Since we will be doing Morse theory on $X$, we will need to understand homotopy properties of links in $X$. In this section we model the links, and then discuss some important subcomplexes of $X$.

\subsection{(General) matching complexes}\label{sec:matching}

In this subsection we establish a useful model for vertex links in $X$.

\begin{definition}
 Let $\Delta$ be a simplicial complex. A \emph{general matching} is a subset $\mu$ of $\Delta$ such that any two simplices in $\mu$ are disjoint. The set of all general matchings, ordered by inclusion, is a simplicial complex, which we call the \emph{general matching complex} $\gmatch(\Delta)$. For $k\in\N_0$ a \emph{$k$-matching} is a general matching that consists only of $k$-cells. The set of all $k$-matchings forms the \emph{$k$-matching complex}. If $\Delta$ is a graph, its $1$-matching complex is the classical \emph{matching complex}, denoted by $\match(\Delta)$.
\end{definition}

By $L_n$ we denote the linear graph on $n$ vertices. Label the vertices $v_1\dots,v_n$ and the edges $e_{1,2},\dots,e_{n-1,n}$, so $e_{i,i+1}$ has $v_i$ and $v_{i+1}$ as endpoints ($1\le i\le n-1$).

\begin{lemma}\label{lem:linear_matching}
 $\match(L_n)$ is $(\lfloor\frac{n-2}{3}\rfloor-1)$-connected.
\end{lemma}

\begin{proof}
 For $n\ge 2$, $\match(L_n)$ is non-empty. Now let $n\ge 5$. Note that $\match(L_n)=\st(\{e_{n-2,n-1}\})\cup\st(\{e_{n-1,n}\})$, a union of contractible spaces, with intersection $\st(\{e_{n-2,n-1}\})\cap\st(\{e_{n-1,n}\})\cong \match(L_{n-3})$. The result therefore follows from induction.
\end{proof}

\medskip

It turns out that links of vertices in the Stein--Farley complex $X$ are general matching complexes. Let $x$ be a vertex of $X$ with $f(x)=n$, where $f$ is the ``number of feet'' function from Observation~\ref{obs:feet_affine}. The cofaces of $x$ are precisely the cells $\sigma = x\Psi$, for every  $\Psi$ such that $x\Psi$ makes sense. In particular, if $\Psi=\gen{A_1,\dots,A_r}$ for $A_i\in\{\elnothing,\elsplit,\elmerge\}$ ($1\le i\le r$), then the rule is that $n$ must equal the number of $A_i$ that are $\elnothing$ or $\elsplit$, plus twice the number that are $\elmerge$.

\begin{observation}\label{obs:vertex_link}
 If a vertex $x$ has $f(x)=n$ feet then $\lk x \cong \gmatch(L_n)$.
\end{observation}

\begin{proof}
 The correspondence identifies a simplex $x\Psi$, for $\Psi=\gen{A_1,\dots,A_r}$, with a matching where (from left to right) $\elnothing$ corresponds to a vertex not in the matching, $\elsplit$ corresponds to a vertex in the matching and $\elmerge$ corresponds to an edge in the matching.
\end{proof}

As a remark, under the identification $\lk x \cong \gmatch(L_n)$, the part of $\lk x$ corresponding to the matching complex $\match(L_n)$ is the \emph{descending link with respect to $f$}. The higher connectivity properties of these descending links are crucial to proving that $F$ is of type $\F_\infty$ using $X$ (originally proved by Brown and Geoghegan \cite{brown84} using a different space).

\subsection{Restricting number of feet}\label{sec:restrict_feet}

Recall the subcomplex $X_{p \le f \le q}$ from Proposition~\ref{prop:char_morse}. This is the full cubical subcomplex of $X$ supported on the vertices $x$ with $p \le f(x) \le q$. Similar notation is applied to define related complexes (e.g., where one inequality is strict or is missing).

\begin{observation}\label{obs:cocompact}
 For $p,q\in\N$, the action of $F$ on $X_{p \le f \le q}$ is cocompact.
\end{observation}

\begin{proof}
 For each $n$, $F$ acts transitively on vertices with $n$ feet. The result follows since $X$ is locally compact.
\end{proof}

\begin{lemma}\label{lem:sf_interval_connectivity}
 The complex $X_{p \le f \le q}$ is $\min(\floor{\frac{q-1}{3}} -1, q-p-1)$-connected. In particular, $X_{p \le f \le q}$ is $(\floor{\frac{q-1}{3}} -1)$-connected for any $p\le \ceil{\frac{2q+1}{3}}$.
\end{lemma}

\begin{proof}
 We first show that $X_{p \le f}$ is contractible for every $p$. Since we know that $X$ is contractible, it suffices to show that the ascending link with respect to $f$ is contractible for every vertex $x$ of $X$. We can then apply the Morse Lemma (using the Morse function $(f,0)$). Indeed, the ascending link is an $(f(x)-1)$-simplex spanned by the cube $x\Phi$ where $\Phi = \gen{\elsplit,\ldots,\elsplit}$. In particular, it is contractible.
 
 Now we filter $X_{p \le f}$ by the spaces $X_{p \le f \le q}$, and so have to study descending links with respect to $f$. The descending link in $X$ of a vertex $x$ with $f(x) > q$ is isomorphic to $\match(L_{f(x)})$, which is $(\floor{\frac{q-1}{3}}-1)$-connected by Lemma~\ref{lem:linear_matching}. But in $X_{p \le f}$ only the $(f(x)-p-1)$-skeleton of that link is present. The descending link is therefore $\min(\floor{\frac{q-1}{3}}-1,q-p-1)$-connected.
\end{proof}

\begin{corollary}
 $X_{2k+1\le f \le 3k+1}$ is $(k-1)$-connected for every $k$.
\end{corollary}


\section{The long interval}\label{sec:long_interval}

We are now ready to compute $\Sigma^m(F)$ for all $m$, using the action of $F$ on $X$. In this and the following two sections, we focus on different parts of the character sphere $S(F)$.

Let $\chi=a\chi_0+b\chi_1$ be a non-trivial character of $F$. In this section we consider the case when $a<0$ or $b<0$. The corresponding part of $S(F)=S^1$ was termed the ``long interval'' in \cite{bieri10}. By symmetry we may assume $a\le b$ (so $a<0$). We will show that for any $m\in\N$, $[\chi]\in\Sigma^m(F)$.

Let
\[
n\defeq 3m+4 \text{.}
\]
Let $X_{2\le f\le n}$ be the sublevel set of $X$ supported on vertices $x$ with $2\le f(x)\le n$. This is $(\lfloor\frac{n-1}{3}\rfloor - 1)$-connected, by Lemma~\ref{lem:sf_interval_connectivity}, and hence is $m$-connected. It is also $F$-cocompact by Observation~\ref{obs:cocompact}. Thus, by Corollary~\ref{cor:changing_filtrations}, to show that $[\chi]\in\Sigma^m(F)$, it suffices to show that the filtration $(X_{2\le f\le n}^{t \le \chi})_{t\in\R}$ is essentially $(m-1)$-connected. To do this, we use the function
$$(\chi,-f) \colon X_{2\le f\le n} \to \R \times \R \text{.}$$
This is a Morse function by Proposition~\ref{prop:char_morse}. By the Morse Lemma, the following lemma suffices to prove that in fact $X_{2\le f\le n}^{t \le \chi}$ is $(m-1)$-connected for all $t\in\R$.

\begin{lemma}
 Let $x$ be a vertex in $X_{2\le f\le n}$. Then the ascending link $\lk^{(\chi,-f)\uparrow}_{X_{2\le f\le n}} x$ in $X_{2\le f\le n}$ is $(m-1)$-connected.
\end{lemma}

\begin{proof}
 Let $L\defeq \lk^{(\chi,-f)\uparrow}_{X_{2\le f\le n}} x$. Vertices of $L$ are obtained from $x$ in two ways: either by adding a split to a foot or a merge to two adjacent feet. For such a vertex to be ascending, in the first case the split must strictly increase $\chi$ and in the second case the merge must not decrease $\chi$. Also note that we only have simplices whose corresponding cubes lie in $X_{2\le f\le n}$. For instance if $f(x)=n$ then the vertices of $L$ may only be obtained from $x$ by adding merges.
 
 We first consider the case when $f(x)=n$, so $L$ is simply the subcomplex of $\match(L_n)$ consisting of matchings that do not decrease $\chi$. Since $a<0$ (and $n>2$), merging the first and second feet decreases $\chi$. Merging the $(n-1)$st and $n$th feet can either decrease, increase, or preserve $\chi$, depending on whether $b<0$, $b>0$ or $b=0$. Any other merging preserves $\chi$ and increases $-f$. Hence $L$ is either $\match(L_{n-1})$ or $\match(L_{n-2})$, and in either case is $(\lfloor\frac{n-4}{3}\rfloor-1)$-connected (Lemma~\ref{lem:linear_matching}) and hence $(m-1)$-connected.
 
 The second case is when $f(x)<n-1$. Thinking of $L$ as the subcomplex of $\gmatch(L_{f(x)})$ supported on vertices with $(\chi,-f)$-value larger than $(\chi,-f)(x)$, since $a<0$ we know $\{e_{1,2}\}\not\in L$. Also, the only vertices of $\gmatch(L_{f(x)})$ of the form $\{v_i\}$ that are in $L$ are $\{v_1\}$ and possibly $\{v_{f(x)}\}$; since $f(x)<n-1$ this implies that for any $\mu\in L$ we have $\mu\cup\{v_1\}\in L$. Hence $L$ is contractible with cone point $\{v_1\}$.
 
 It remains to consider the case when $f(x)=n-1$. Now it is convenient to define
 \[
 L'\defeq\lk^{(\chi,-f)\uparrow}_{X_{2\le f\le n+1}} x \text{.}
 \]
 By the previous paragraph, $L'$ is contractible, via coning off $\{v_1\}$. If $b\ge0$, so $\{v_{n-1}\}\not\in L$, then we can similarly cone $L$ off on $\{v_1\}$. The problem is when $b<0$, since then $\{v_1\},\{v_{n-1}\}\in L$, but $\{v_1,v_{n-1}\}\not\in L$ since this involves $n+1$ feet. However, this is the \emph{only} simplex of $L'$ that is not in $L$. Hence we can remove $\{v_1,v_{n-1}\}$ from the contractible complex $L'$ along its relative link, and get $L$. Said relative link is the join of the boundary $\{v_1\}\cup\{v_{n-1}\} \simeq S^0$ with the matching complex $\match(L_{n-3})$. This relative link is $\lfloor\frac{n-5}{3}\rfloor$-connected (Lemma~\ref{lem:linear_matching}), so $L$ is as well, and in particular is $(m-1)$-connected.
\end{proof}

The lemma together with the Morse Lemma gives:

\begin{proposition}
 If $a< 0$ or $b<0$ then $[\chi]\in\Sigma^\infty(F)$.
\end{proposition}


\section{The characters $\chi_0$ and $\chi_1$}\label{sec:chi0_and_chi1}

Let $\chi=a\chi_0+b\chi_1$ be a non-trivial character of $F$ with $a > 0$ and $b = 0$, or $b > 0$ and $a = 0$. In this section we show that $[\chi] \not\in \Sigma^1(F)$. We do this by considering the Morse function $(\chi,f)\colon X_{3 \le f \le 4} \to \R\times\R$. Thanks to Lemma~\ref{lem:sf_interval_connectivity} and Observation~\ref{obs:cocompact}, $X_{3 \le f \le 4}$ is connected and $F$-cocompact. Hence by Corollary~\ref{cor:changing_filtrations} it suffices to show that the filtration $(X_{3 \le f \le 4}^{t \le \chi})_{t\in\R}$ is not essentially connected. We would like to apply Observation~\ref{obs:weak_bottleneck_trick}, for which we need the following:

\begin{lemma}\label{lem:asc_links_3_4}
 In $X_{3 \le f \le 4}$ all $(\chi, f)$-ascending vertex links are non-empty.
\end{lemma}

\begin{proof}
 Write $Y \defeq X_{3 \le f \le 4}$. Up to scaling and symmetry we may assume that $\chi = \chi_0$. Given a vertex $v$ on $n = f(v)$ feet, a simplex in $\lk_X v \cong \gmatch(L_n)$ is ascending if and only if its vertices are among $e_{1,2}$ ($\chi$-ascending) and $v_2, \ldots, v_n$ ($f$-ascending). The number of feet imposes restrictions on which of these simplices actually lie in $\lk_Y x$. But any ascending link contains $e_{1,2}$ or $v_2$ and thus is non-empty.
\end{proof}

\begin{proposition}\label{prop:neg_zero_connectivity}
 If $a > 0$ and $b = 0$ or $b > 0$ and $a = 0$, then $[\chi] \not\in \Sigma^1(F)$.
\end{proposition}

\begin{proof}
 We treat the case $a > 0$, $b = 0$; the other case follows from symmetry. By positive scaling we can assume $a=1$, so $\chi = \chi_0$. We want to show that $(X_{3 \le f \le 4}^{t \le \chi_0})_{t\in\R}$ is not essentially connected. By Observation~\ref{obs:weak_bottleneck_trick} and Lemma~\ref{lem:asc_links_3_4} it suffices to show that $X_{3 \le f \le 4}^{t \le \chi_0}$ is not connected for any $t \in \Z$. Since $F$ acts transitively on these sets, this is equivalent to proving that $X_{3 \le f \le 4}^{0 \le \chi_0}$ is not connected.
 
 Let $x$ be a vertex in $X$, and let $(T/E)$ be its unique reduced representative diagram. Define $L(x)\defeq L(T)$. Recall that $L(T)$ is the number of carets above the leftmost leaf of $T$. Let $x'$ be any neighbor of $x$, say with reduced diagram $(T'/E')$, so $L(x')\defeq L(T')$. Note that $L(x')\in\{L(x)-1,L(x),L(x)+1\}$. We claim that if these neighboring vertices $x,x'$ are in $X_{3 \le f \le 4}^{0 \le \chi_0}$, then $L(x)=L(x')$.
 
 First note that $L(E)\ge L(x)$ because $\chi_0(x)\ge0$. Also, $L(x)\ge1$ since $f(x)\ge3$, so also $L(E)\ge1$. Now if $x'$ is obtained from $x$ by adding a merge to the feet of $x$, then $T$ and $T'$ have the same number of carets above the leftmost leaf, so $L(x)=L(x')$. Alternately, if $x'$ is obtained from $x$ by adding a split, since the leftmost leaf of $E$ has at least one caret above it, we know that, again, adding this split cannot change the number of splits  above the leftmost leaf of $T$. Hence in any case $L(x')=L(x)$ for all neighbors $x'$ of $x$.
 
 This shows that $L$ is constant along the vertices of any connected component of $X_{3 \le f \le 4}^{0 \le \chi_0}$. Since $L$ does take different values, there must be more than one component.
\end{proof}


\section{The short interval}\label{sec:short_interval}

Let $\chi=a\chi_0+b\chi_1$ be a non-trivial character of $F$, with $a > 0$ and $b > 0$. The corresponding part of $S(F)=S^1$ was termed the ``short interval'' in \cite{bieri10}. In this section we show that $[\chi]\in\Sigma^1(F)\setminus \Sigma^2(F)$. Consider the Morse function $(\chi,f)\colon X_{4 \le f \le 7} \to \R\times\R$. By Lemma~\ref{lem:sf_interval_connectivity} and Observation~\ref{obs:cocompact}, $X_{4 \le f \le 7}$ is simply connected and $F$-cocompact. Hence by Corollary~\ref{cor:changing_filtrations} it suffices to show that the filtration $(X_{4 \le f \le 7}^{t \le \chi})_{t\in\R}$ is essentially connected but not essentially simply connected. We would like to apply Observation~\ref{obs:weak_bottleneck_trick}, for which we need the following:

\begin{lemma}\label{lem:asc_links_4_7}
 In $X_{4 \le f \le 7}$ all $(\chi, f)$-ascending links are connected.
\end{lemma}

\begin{proof}
 Given a vertex $x$ with $n = f(x)$ feet, a simplex in $\lk_X x \cong \gmatch(L_n)$ is ascending if and only if its vertices all lie in $\{e_{1,2}, e_{n-1,n}, v_2,\dots,v_{n-1}\}$. The number of feet imposes restrictions on which of these simplices actually lie in $\lk_{X_{4 \le f \le 7}} x$.
 
 We claim that all ascending links are connected. For $n\in\{4,5\}$, $v_2$ through $v_{n-1}$ are pairwise connected by an edge, and each of $e_{1,2}$ and $e_{n-1,n}$ is connected to at least one of them. For $n\in\{6,7\}$ the $0$-simplices $e_{1,2}$ and $e_{n-1,n}$ are connected by an edge, and each $v_i$ is connected to at least one of them. Thus in either case the ascending link is connected.
\end{proof}

To show that superlevel sets are not simply connected we will use the following supplement to the Nerve Lemma.

\begin{lemma}\label{lem:alternative_nerve_lemma}
 Let a simplicial complex $Z$ be covered by connected subcomplexes $(Z_i)_{i \in I}$. Suppose the nerve $\nerve(\{Z_i\}_{i \in I})$ is connected but not simply connected. Then $Z$ is connected but not simply connected.
\end{lemma}

\begin{proof}
 Let $N \defeq \nerve(\{Z_i\}_{i \in I})$. For each $i \in I$ pick a vertex $z_i \in Z_i$. For any edge $\{i,j\}$ in $N$, pick an edge path $p_{i,j}$ from $z_i$ to $z_j$ in $Z_i \cup Z_j$ (this is possible because $Z_i$ and $Z_j$ are connected and meet non-trivially). This induces a homomorphism $\varphi \colon \pi_1(N,i) \to \pi_1(Z,z_i)$ where $i \in I$ is arbitrary. We want to see that $\varphi$ is injective.
 
 To do so, let $\{i_0,i_1\}, \ldots , \{i_{n-1},i_n\}, \{i_n,i_0\}$ be an edge cycle $\gamma$ in $N$ and let $v_0, \ldots, v_n$ be any sequence of vertices $v_j \in Z_{i_j}$ such that the edge $\{v_j,v_{j+1}\}$ exists in $Z_{i_j} \cup Z_{i_{j+1}}$ (with subscripts taken mod $n+1$), so these form an edge cycle $c$ in $Z$. Here we allow for degenerate edges in either of these edge cycles, i.e., each edge may actually be a vertex. The condition that $\{v_j,v_{j+1}\} \subseteq Z_{i_j} \cup Z_{i_{j+1}}$ for all $j$ will be referred to as $\gamma$ and $c$ being \emph{linked}. We assume that $c$ can be filled by a triangulated disk in $Z$ and want to show that $\gamma$ is nullhomotopic in $N$.
 
 The proof is by induction on the number of triangles in a (minimal) such filling disc. Let $t \subseteq Z$ be a triangle of a filling disk with vertices $v_j$, $v_{j+1}$, and $w$, say. Let $k$ be such that $Z_k$ contains $t$. Then we obtain a path homotopic to $c$ by replacing the edge $\{v_j,v_{j+1}\}$ by the union of edges $\{v_j,w\}$ and $\{w,v_{j+1}\}$, and we obtain a path homotopic to $\gamma$ by replacing the edge $\{i_j,i_{j+1}\}$ by the edges $\{i_j,k\}$ and $\{k,i_{j+1}\}$. Note that $\gamma$ and $c$ remain linked after this process, since $\{v_j,w\}$ and $\{w,v_{j+1}\}$ lie in $Z_k$. After finitely many such reductions, the filling disc for $c$ contains no triangles.
 
 To reduce $c$ to a trivial cycle, the only remaining reductions needed are removing two forms of stuttering: the one where $v_j = v_{j+1}$ and the one where $v_j = v_{j+2}$. In the second case, we may remove $v_{j+1}$ from $c$ to obtain a homotopic path and we may similarly remove $i_{j+1}$ from $\gamma$ to obtain a (linked) homotopic path (note that $i_j$, $i_{j+1}$ and $i_{j+2}$ span a triangle because $Z_{i_j}\cap Z_{i_{j+1}}\cap Z_{i_{j+2}}$ contains either $v_j$ or $v_{j+1}$, thanks to $\gamma$ and $c$ being linked). This reduces the second kind of stuttering to the first. The first kind of stuttering can be resolved if $Z_{j+2}$ contains $v_{j+1}$ by just deleting $v_{j+1}$ and $i_{j+1}$ from their respective paths. Otherwise ($Z_{j+1}$ contains $v_{j+2}$ and) the stuttering can be shifted by replacing $v_{j+1}$ by $v_{j+2}$ in $c$. Under any of these moves $\gamma$ and $c$ remain linked. After finitely many such reductions both $c$ and $\gamma$ will be trivial paths.
 
 It follows that $\varphi$ is injective: Let $\gamma$ be a cycle in $N$ and let $c$ be the corresponding cycle in $Z$ made up of the relevant paths $p_{k,\ell}$, as in the definition of $\varphi$. Each $p_{k,\ell}$ has a linked path consisting of $\{k,\ell\}$ and (possibly many occurrences of) $\{k\}$ and $\{\ell\}$. Thus $\gamma$ and $c$ are linked after sufficiently many degenerate edges have been added to $\gamma$. By the above argument, if $\varphi([\gamma])=[c] \in \pi_1(Z,z_i)$ is trivial then so is $[\gamma] \in \pi_1(N,i)$.
\end{proof}

\begin{proposition}\label{prop:neg_one_connectivity}
 If $a > 0$ and $b > 0$  then $[\chi] \in \Sigma^1(F) \setminus \Sigma^2(F)$.
\end{proposition}

\begin{proof}
 We want to show that the filtration $(X_{4 \le f \le 7}^{t \le \chi})_{t\in\R}$ is essentially connected but not essentially simply connected. First note that the Morse Lemma together with Lemma~\ref{lem:asc_links_4_7} shows that in fact every $X_{4 \le f \le 7}^{t \le \chi}$ is already connected.
 
 To show that $(X_{4 \le f \le 7}^{t \le \chi})_{t\in\R}$ is not essentially simply connected, it suffices by Observation~\ref{obs:weak_bottleneck_trick} and Lemma~\ref{lem:asc_links_4_7} to show that $X_{4 \le f \le 7}^{t \le \chi}$ fails to be simply connected, for arbitrarily small $t \in \R$. Since the action of $F$ on $\R$ induced by $\chi$ is cocompact, it suffices to show that $X_{4 \le f \le 7}^{0 \le \chi}$ is not simply connected.
 
 To this end we consider certain subcomplexes of $Y\defeq X_{4 \le f \le 7}^{0 \le \chi}$. To define them, as in the proof of Proposition~\ref{prop:neg_zero_connectivity}, for a vertex $x$ of $X$ with reduced representative $(T/E)$, define $L(x)\defeq L(T)$. Similarly let $R(x)\defeq R(T)$. Now we consider the full subcomplexes $Y_{L=i}$ of $Y$ supported on vertices $x$ with $L(x) = i$. Similarly we have complexes $Y_{R=i}$. Each of these spaces decomposes into countably many connected components, which we enumerate as $Y_{B=i}^m$, $m \in \N$ (where $B=L$ or $B=R$).
 
 We claim that $Y$ is covered by the complexes $Y_{L=i}^m, m \in \N$ and $Y_{R=i}^m, m \in \N$. We have to show that every cell $x\Psi$ in $Y$ is either contained in some $Y_{L=i}$ or some $Y_{R=i}$. Take $x$ to have the maximal number of feet within the cell so that $\Psi$ involves only $\elnothing$ and $\elmerge$. Note that in $(T/E)$, which is the reduced representative of $x$, we know that $T$ is non-trivial since $f(x)\ge4$, and so $L(T)>0$ and $R(T)>0$. Hence, in order for $\chi(x) \ge 0$ to hold we need at least one of $L(E)>0$ or $R(E)>0$ to hold. If $L(E)>0$, then $L(y)=L(x)$ for all vertices $y$ of $x\Psi$ (for similar reasons as in the proof of Proposition~\ref{prop:neg_zero_connectivity}), with a similar statement for the other case, and hence $x\Psi$ is contained in some $Y_{B = i}$.
 
 Note that by definition $Y_{B = i} \cap Y_{B = j} = \emptyset$ for $i \ne j$ and $B \in \{L,R\}$. Thus the nerve $\nerve(\{Y_{B=i}^m\}_{i,m \in \N, B \in \{L,R\}})$ is a bipartite graph. That it is connected can be deduced from \cite[Lemma~1.2]{bjoerner94} since $Y$ is connected.
 
 To construct an explicit cycle, consider the vertices $x_1$ to $x_4$ in Figure~\ref{fig:vine_splerges}. Note that $x_4$ and $x_1$ lie in the same component of $Y_{L=2}$. This is because after extending the left hand side by splitting the second foot and then merging the first two feet sufficiently many times (depending on $a$ and $b$), $\chi_0$ is so high that the right side can be removed without $\chi$ dropping below zero. Let us say that this component is $Y_{L=2}^0$. For similar reasons $x_1$ and $x_2$ lie in a common component $Y_{R=2}^0$, the vertices $x_2$ and $x_3$ lie in a common component $Y_{L=3}^0$ and $x_3$ and $x_4$ lie in a common component $Y_{R=2}^0$. It follows that $\nerve(\{Y_{B=i}^m\}_{i,m \in \N, B \in \{L,R\}})$ contains the cycle $Y_{L=2}^0, Y_{R=2}^0, Y_{L=3}^0, Y_{R=3}^0$. Now Lemma~\ref{lem:alternative_nerve_lemma} tells us that $Y$ is not simply connected (it applies to simplicial complexes, so we take a simplicial subdivision).
\end{proof}

\begin{figure}[htb]
\centering
\hspace{.5cm}
\includegraphics[scale=.5]{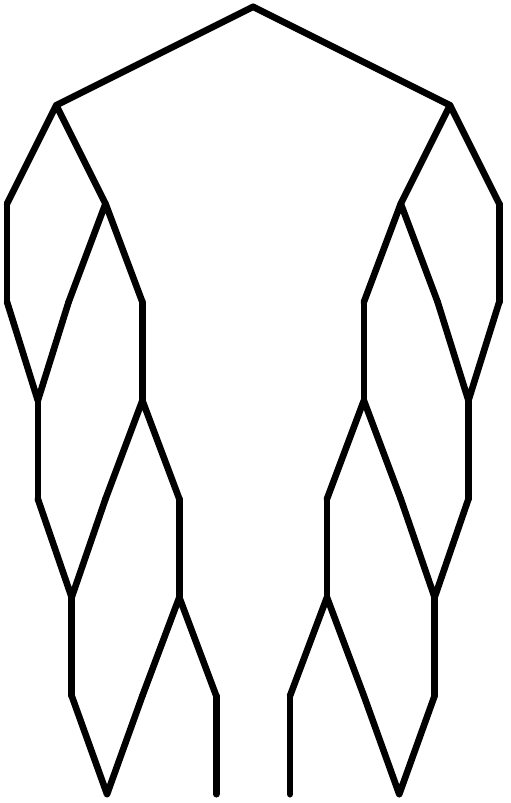}
\hfill
\includegraphics[scale=.5]{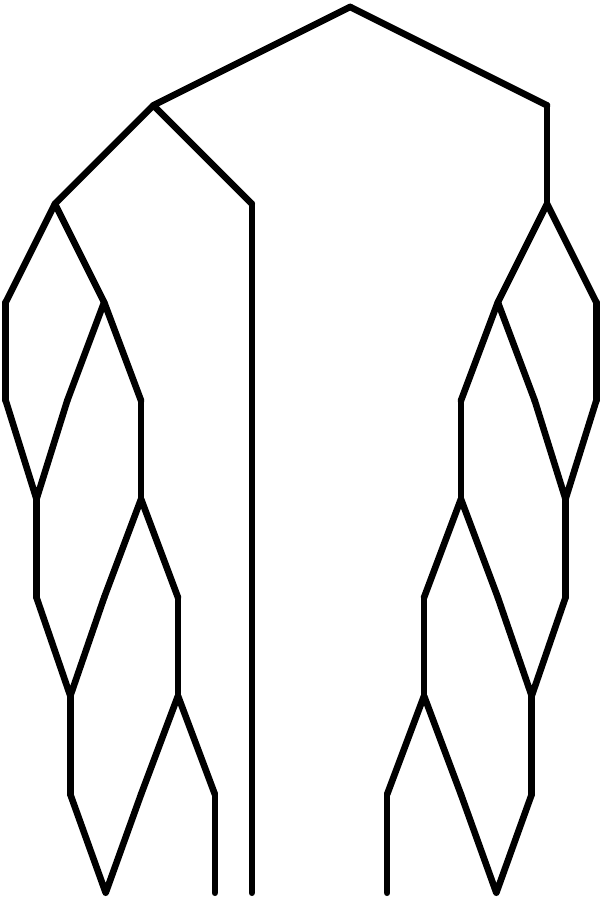}
\hfill
\includegraphics[scale=.5]{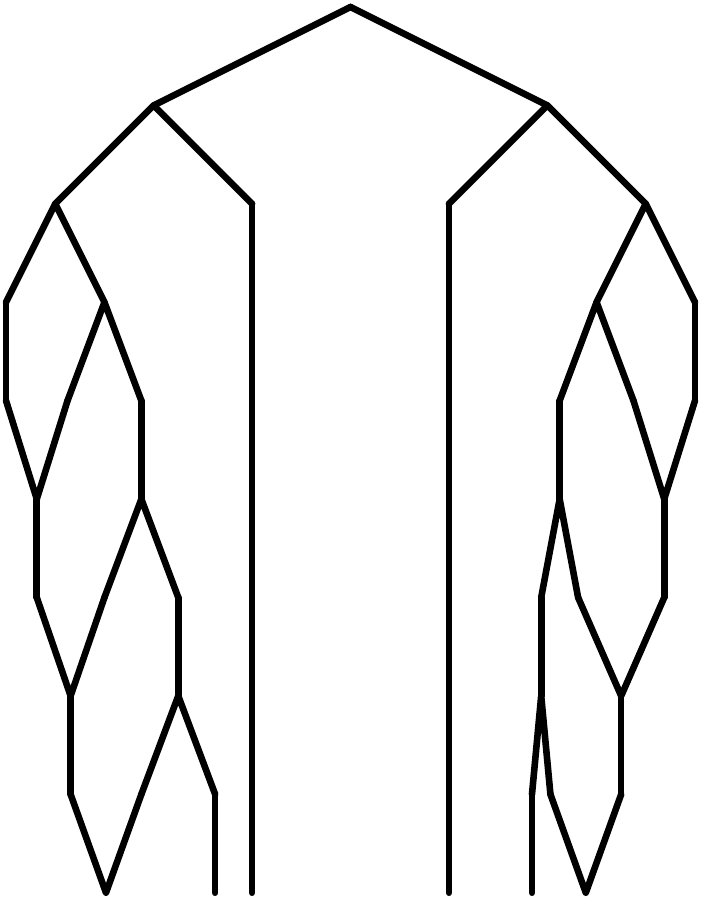}
\hfill
\includegraphics[scale=.5]{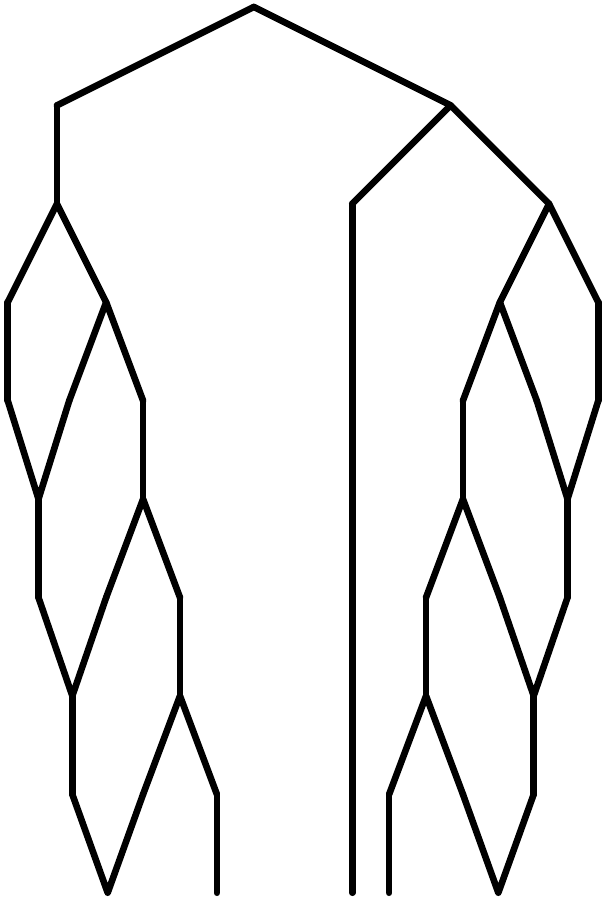}
\hspace{.5cm}{}

\hspace{1.6cm} $x_1$ \hfill $x_2$ \hfill $x_3$ \hfill $x_4$ \hspace{1.6cm}{}
\caption{The diagrams used in the proof of Proposition~\ref{prop:neg_one_connectivity}.}
\label{fig:vine_splerges}
\end{figure}

\providecommand{\bysame}{\leavevmode\hbox to3em{\hrulefill}\thinspace}
\providecommand{\MR}{\relax\ifhmode\unskip\space\fi MR }
\providecommand{\MRhref}[2]{%
  \href{http://www.ams.org/mathscinet-getitem?mr=#1}{#2}
}
\providecommand{\href}[2]{#2}


\begin{thebibliography}{BLV{\v{Z}}94}

\bibitem[BB97]{bestvina97}
M.~Bestvina and N.~Brady, \emph{Morse theory and finiteness properties of
  groups}, Invent. Math. \textbf{129} (1997), no.~3, 445--470.

\bibitem[BG84]{brown84}
K.~S. Brown and R.~Geoghegan, \emph{An infinite-dimensional torsion-free {${\rm
  FP}_{\infty }$} group}, Invent. Math. \textbf{77} (1984), no.~2, 367--381.

\bibitem[BG03]{bieri03}
R.~Bieri and R.~Geoghegan, \emph{Connectivity properties of group actions on
  non-positively curved spaces}, Mem. Amer. Math. Soc. \textbf{161} (2003),
  no.~765.

\bibitem[BGK10]{bieri10}
R.~Bieri, R.~Geoghegan, and D.~H. Kochloukova, \emph{The {S}igma invariants of
  {T}hompson's group {$F$}}, Groups Geom. Dyn. \textbf{4} (2010), no.~2,
  263--273.

\bibitem[BH99]{bridson99}
M.~R. Bridson and A.~Haefliger, \emph{Metric spaces of non-positive curvature},
  Die Grundlehren der Mathematischen Wissenschaften, vol. 319, Springer, 1999.

\bibitem[BLV{\v{Z}}94]{bjoerner94}
A.~Bj{\"o}rner, L.~Lov{\'a}sz, S.~T. Vre{\'c}ica, and R.~T.
  {\v{Z}}ivaljevi{\'c}, \emph{Chessboard complexes and matching complexes}, J.
  London Math. Soc. (2) \textbf{49} (1994), no.~1, 25--39.

\bibitem[BNS87]{bieri87}
R.~Bieri, W.~D. Neumann, and R.~Strebel, \emph{A geometric invariant of
  discrete groups}, Invent. Math. \textbf{90} (1987), no.~3, 451--477.

\bibitem[BR88]{bieri88}
R.~Bieri and B.~Renz, \emph{Valuations on free resolutions and higher geometric
  invariants of groups}, Comment. Math. Helv. \textbf{63} (1988), no.~3,
  464--497.

\bibitem[Bux04]{bux04}
K.-U. Bux, \emph{Finiteness properties of soluble arithmetic groups over global
  function fields}, Geom. Topol. \textbf{8} (2004), 611--644.

\bibitem[CFP96]{cannon96}
J.~W. Cannon, W.~J. Floyd, and W.~R. Parry, \emph{Introductory notes on
  {R}ichard {T}hompson's groups}, Enseign. Math. (2) \textbf{42} (1996),
  no.~3-4, 215--256.

\bibitem[Far03]{farley03}
D.~S. Farley, \emph{Finiteness and {$\rm CAT(0)$} properties of diagram
  groups}, Topology \textbf{42} (2003), no.~5, 1065--1082.

\bibitem[Geo08]{geoghegan08}
R.~Geoghegan, \emph{Topological methods in group theory}, Graduate Texts in
  Mathematics, vol. 243, Springer, New York, 2008.

\bibitem[Ste92]{stein92}
M.~Stein, \emph{Groups of piecewise linear homeomorphisms}, Trans. Amer. Math.
  Soc. \textbf{332} (1992), no.~2, 477--514.

\end{thebibliography}
\end{document}